\documentclass[11pt]{amsproc}

\usepackage{amsmath, amssymb, amsthm, latexsym,nicefrac,setspace}
\usepackage[pagebackref,colorlinks,linkcolor=red,citecolor=blue,urlcolor=blue,hypertexnames=true]{hyperref}
\usepackage[backrefs]{amsrefs}

\theoremstyle{plain}
\newtheorem{theorem}{Theorem}[section]

\newtheorem{lemma}[theorem]{Lemma}
\newtheorem{corollary}[theorem]{Corollary}
\newtheorem{proposition}[theorem]{Proposition}

\newtheorem{conjecture}[theorem]{Conjecture}

\theoremstyle{remark}
\newtheorem{remark}[theorem]{Remark}
\newtheorem*{theorem*}{Theorem}

\input{xy}
\xyoption{all}

\def\N{\mathbb{N}}

\DeclareMathOperator{\Aut}{Aut}

\title[Length of shortest non-trivial elements]{A new construction for the shortest non-trivial element in the lower central series}
\author{Abdelrhman Elkasapy}
\address{Abdelrhman Elkasapy, Max Planck institute for Mathematics in the Sciences, Inselstra\ss e 22,
04103 Leipzig, Germany, and Mathematics Department, South Valley University, Qena, Egypt}
\email{elkasapy@mis.mpg.de}

\begin{document}

\onehalfspace

\begin{abstract}
We provide a new upper bound for the length for the shortest non-trivial element in the lower central series $\gamma_n(\mathbb{F}_2)$ of the free group on two generators. We prove that it has an asymptotic behaviour of the form $O(n^{\log_{\varphi}(2)})$, where $\varphi=1.618...$ is the golden ratio.  This new technique is used  to provide new estimates on the length of laws for finite groups and on almost laws for compact groups.
\end{abstract}

\maketitle

\tableofcontents

\section{Introduction}
This is a companion paper to
\cite{Abdu} and we will refer to it for most of the definitions and some of the results. It has been conjectured by Malestein and Putman \cite{MR2737679} that the length of the shortest non-trivial element in the derived series and the lower central series in the free group of rank 2 has an asymptotic behaviour of the form $O(n^2)$. The author and A. Thom \cite{Abdu} disproved this conjecture by some construction which provides a better upper bound. In this note we slightly improve this upper bound. The new construction suggests that this asymptotic behaviour is of the form  $O(n^{\log_{\varphi}(2)})$, where $\varphi$ is the golden ratio. The most interesting thing in this new construction is that it is enough to multiply the word length by $2$ to increase the depth of the central series by some Fibonacci number. Indeed, the lower central series steps are  $\gamma_{f_{i+2}}(\mathbb{F}_2)$ where ${(f_n)}_{n\in\mathbb{N}}$ is the sequence of Fibonacci numbers ($f_0:=0$, $f_1:= 1$, and $f_{n+2}:= f_{n+1}+f_n$ for $n\in\mathbb{N}$), see Remark \ref{rem1}. In this paper we conjecture this asymptotic behaviour $O(n^{\log_{\varphi}(2)})$ to be sharp and it is still unclear how to prove this conjecture.  However, we give some remarks on the length of some elements in this construction which may support this conjecture, see Remark \ref{rem2} and Remark \ref{rem3}. Moreover, we also improve some previous results for the length of laws for finite groups, see Corollary \ref{cor1} and Corollary \ref{cor2} and on almost laws for compact groups, see Section \ref{almost}.  
 We write and $f(n) \preceq g(n)$ if there is a constant $C$ such that
$f(n) \leq C g(Cn)$ for all $n \in \N$. The author and A. Thom \cite{Abdu}  studied the growth of the function
$\alpha(n):= \min \{ \ell(w) \mid w \in \gamma_n({\mathbb F}_2) \setminus \{e\} \} $
 where $\ell \colon {\mathbb F}_2 \to \N$  is length function with respect to the generating set $\{a,a^{-1},b,b^{-1}\}$. We can think of $\alpha(n)$  as the girth of the Cayley graph of the group ${\mathbb F}_2/\gamma_n({\mathbb F}_2)$  with respect to the image of the natural generating set of ${\mathbb F}_2$. In view of \cite[Lemma 2.1]{Abdu} we set   $\alpha := \lim_{n \to \infty} \frac{\log_2(\alpha(n))}{\log_2(n)}.$ 
The  author and A. Thom \cite{Abdu} provides upper bound for $\alpha$, that is $\alpha \leq \frac{\log_2(3+\sqrt{17}) - 1}{\log_2(1 + \sqrt{2})}= 1,4411...\ $

For more background on free groups you can see \cite{lyndonschupp} and \cite{lyndonschupp2}. %and %\cite{Dan}.

We provide a new construction of words in the lower central series which slightly improves the upper bound of $\alpha$ in Theorem \cite{Abdu}. Our main result is the following:
\begin{theorem}\label{new}
Let $\mathbb F_2$ be the free group on two generators and $(\alpha(n))_{n \in \N}, \text{and} \ \alpha$  be defined as above. We have
$$\alpha \leq \log_{\varphi}(2)= 1,440..\ $$ or equivalently
$\alpha(n) \preceq n^{\log_{\varphi}(2)+ \varepsilon}$ for all $\varepsilon>0$.
\end{theorem}
We state this conjecture: %that $\alpha={(\log_2 \varphi)}^{-1}$. 
\begin{conjecture}\label{conj}
$\alpha= \log_{\varphi}(2)$.
\end{conjecture}

\section{The new construction} 
\label{construction}
%Recall that $\mathbb{F}_2$ is the free group over the set $\{a,b\}$.
\begin{lemma}\label{depth}
For $w_1, w_2  \in \mathbb{F}_2, \text{and} \ n \in \mathbb{Z},$ we have
\begin{equation}\label{dep}
[w_1,w_2]=[w_1w_2^n,w_2] \ \ \text{and} \ \ [w_1,w_2]=[w_1,w_2w_1^n].
\end{equation}
\end{lemma}
\begin{proof} We compute:
\begin{eqnarray*}
[w_1w_2^n,w_2] = w_1w_2^nw_2w_2^{-n}w_1^{-1}w_2^{-1} 
=w_1w_2w_1^{-1}w_2^{-1} 
= [w_1,w_2].
\end{eqnarray*}
In the same way we show that $[w_1,w_2]=[w_1,w_2w_1^n]$ and this proves the claim.
\end{proof}

%Recall that we consider ${\mathbb F}_2$ to be generated by letters $a$ and $b$.
We set $a_{0}:=b^{-1}, \ \ b_0:=aba^{-1}$ and define recursively
$$a_{n} := a_{n-1}b_{n-1}, \quad b_{n}:=a_{n-1}^{-1}b_{n-1}^{-1}, \quad \text{for all} \ n \in \N.$$ 
%$$w_n=[a_n,b_n].$$

\begin{lemma}[length of $a_n$ and $b_n$]\label{lem:length a_n b_n}
	Let $n>0$. Then $a_n$ and $b_n$ have the following reduced representations:
	$$
		a_n=
		\begin{cases}
			b^{-1}\cdots b^{-1} &: n\equiv 0 \mod 3\\
			b^{-1}\cdots b a^{-1} &: n\equiv 1 \mod 3\\
			b^{-1}\cdots b^{-1}a^{-1} &: n\equiv 2 \mod 3
		\end{cases}
		\quad \text{and} \quad
		b_n=
		\begin{cases}
			a b\cdots b a^{-1} &: n\equiv 0 \mod 3\\
			b\cdots b^{-1}a^{-1} &: n\equiv 1 \mod 3\\
			a b^{-1}\cdots b^{-1} &: n\equiv 2 \mod 3
		\end{cases}.
	$$
	Thus, there is no cancellation in the product $a_n^{-1} b_n^{-1}$ for all $n\in\mathbb{N}$. The product $a_n b_n$ involves cancellation if an only if $n\equiv 2\mod 3$, where the term $a^{-1}a$ cancels out. Hence, for $n\in\mathbb{N}$ we have
	$$
		\ell(a_{n+1})=\ell(a_n)+\ell(b_n)-
		\begin{cases}
			2 &: n\equiv 2 \mod 3\\
			0 &: \text{otherwise}
		\end{cases} 
		\quad \text{and} \quad 
		\ell(b_{n+1})=\ell(a_n)+\ell(b_n).
	$$
	Solving the recursion yields
	$$
		\ell(a_n)=
		\begin{cases}
			\frac{13\cdot 2^n-6}{7} &: n\equiv 0\mod 3\\
			\frac{13\cdot 2^n+2}{7} &: n\equiv 1\mod 3\\
			\frac{13\cdot 2^n+4}{7} &: n\equiv 2\mod 3
		\end{cases}
		\quad \text{and} \quad
		\ell(b_n)=
		\begin{cases}
			\frac{13\cdot 2^n+8}{7} &: n\equiv 0\mod 3\\
			\frac{13\cdot 2^n+2}{7} &: n\equiv 1\mod 3\\
			\frac{13\cdot 2^n+4}{7} &: n\equiv 2\mod 3
		\end{cases}.
	$$
	 In particular, there exists a constant $C'>0$, such that $\ell(a_n) \leq C' \cdot 2^n$ for all $n$.
\end{lemma}

\begin{proof}
At first we mention why it is enough to show the statement about the reduced representations fo $a_n$ and $b_n$ (for $n>0$).

Namely, from these formulas it follows directly that there is no cancellation in the product $a_n^{-1}b_n^{-1}$ and the described cancellation in $a_n b_n$ (for $n=0$ one checks this explicitly). From this we deduce the stated recursion formulas for $\ell(a_n)$ and $\ell(b_n)$. A straightforward induction shows the correctness of the given explicit formulas for $\ell(a_n)$ and $\ell(b_n)$.

Finally, let us prove the statement about the reduced representations of $a_n$ and $b_n$ by induction on $n\in\mathbb{N}_{>0}$. For $n=1$ we have $a_1=b^{-1}a b a^{-1}$ and $b_1=b a b^{-1}a^{-1}$, so the claim is true. Now let us do one step of the induction, e.g.~the case $n\equiv 0 \mod 3$.
Then $a_{n+1}=b^{-1}\cdots b^{-1} a b \cdots b a^{-1}=b^{-1}\cdots b a^{-1}$ and $b_{n+1}=b\cdots b a b^{-1}\cdots b^{-1} a^{-1}=b\cdots b^{-1} a^{-1}$ as claimed, since $n+1\equiv 1 \mod 3$. The cases $n\equiv 1,2 \mod 3$ are treated analogously.
\end{proof}

We set 
\begin{equation} 
\gamma(w):=\max\{n \mid w \in \gamma_n({\mathbb F}_2)\}. \quad \forall w \in \mathbb F_2.
\end{equation} Clearly, 
%\begin{equation}
%\gamma(w_1w_2) \geq \min\{\gamma(w_1),\gamma(w_2)\} \ \text{and} \
%\end{equation} 
\begin{equation} \label{eq}
\gamma([w_1,w_2]) \geq \gamma(w_1) + \gamma(w_2),
\end{equation}
and
\begin{equation}\label{alpha}
\alpha(\gamma(a_n)) \leq \ell(a_n) \ \text{for all} \  n \in \N.
\end{equation}
%-----------

\begin{lemma}\label{lem:rel a_n b_n}
	Let $\sigma,\tau\in \Aut(\mathbb{F}_2)$ be defined by $\sigma:a\mapsto a^{-1}, b\mapsto b^{-1}$ and $\tau:a\mapsto a, b\mapsto b^{-1}$.
	Let $n\in\mathbb{N}$.
	\begin{enumerate}
		\item If $n\equiv 0 \mod 3$ then $a \sigma(a_n) a^{-1}=b_n$ (and so $a_n=\sigma^{-1}(a^{-1} b_n a)=\sigma(a^{-1} b_n a)=a\sigma(b_n)a^{-1}$).
	\item If $n\equiv 1 \mod 3$ then $\tau(a_n)=b_n$ (and so $\tau(b_n)=\tau^2(a_n)=a_n$).
		\item If $n\equiv 2 \mod 3$ then $\tau(a_n)=b_n^{-1}$.
	\end{enumerate} 
\end{lemma}

\begin{proof}
	At first we prove statements (1) and (2) by induction on $n$. 
	For $n=0$ we have $a_0=b^{-1}\stackrel{a\sigma(\cdot)a^{-1}}{\mapsto} ab a^{-1}=b_0$, and for $n=1$ we have $a_1=b^{-1}aba^{-1}\stackrel{\tau}{\mapsto}bab^{-1} a^{-1}$ verifying the claim. Now, let $\eta=a\sigma(\cdot)a^{-1}$ if $n\equiv 0\mod 3$ and $\eta=\tau$ if $n\equiv 1\mod 3$. By definition it holds for $n\in\mathbb{N}$ that
	$$
	a_{n+3}=[a_n b_n, a_n^{-1} b_n^{-1}], 
	\quad \text{and} \quad 
	b_{n+3}=[b_n a_n, b_n^{-1} a_n^{-1}].
	$$
	Thus, the induction hypothesis gives us $b_{n+3}=[\eta(a_n)\eta(b_n),{\eta(a_n)}^{-1}{\eta(b_n)}^{-1}]=\eta([a_n b_n,a_n^{-1} b_n^{-1}])=\eta(a_{n+3})$.
	
	At last we show statement (3). Let $n\equiv 2\mod 3$. Then, by statement (2) and by definition we obtain $b_n^{-1}=b_{n-1}a_{n-1}=\tau(a_{n-1})\tau(b_{n-2})=\tau(a_{n-1}b_{n-1})=\tau(a_n)$.
\end{proof}

An immediate corollary of the last lemma is

\begin{corollary}\label{cor:gamma(a_n)=gamma(b_n)}
	We have $\gamma(a_n)=\gamma(b_n)$.
\end{corollary}

\begin{proof}
    By Lemma~\ref{lem:rel a_n b_n} we get for a characteristic subgroup $G\subseteq \mathbb{F}_2$ that $a_n\in G$ if
    and only if $b_n\in G$.
\end{proof}

%------------

In the following lemma, we estimate $\gamma(a_n).$

\begin{lemma} \label{lower}
We have $\gamma(a_{n+2}) \geq  \gamma(a_{n+1}) + \gamma(a_{n})$ for all $n \in \N$. In particular, there exists a constant $C>0$, such that
$\gamma(a_n) \geq C \cdot \varphi^n$, where $\varphi=1.618...$ is the golden ratio.
\end{lemma}
\begin{proof}
We compute
%\begin{eqnarray*}
$a_{n+2 }= [a_{n},b_{n}] \stackrel{\eqref{dep}}{=} [a_{n}b_{n}, b_{n}].$ Then it follows from \ref{eq} and  \ref{cor:gamma(a_n)=gamma(b_n)} that $\gamma(a_n) \geq  \gamma(a_{n-1}) + \gamma(a_{n-2})$. 
The estimate on $\gamma(a_n)$ follows from the fact that the golden ratio $\varphi$ is the largest root of the polynomial  $p(t)=t^2 -t-1$.
\end{proof}
Now we are ready to prove Theorem \ref{new}, the proof follows from the following proposition.
\begin{proposition}
 We have $\alpha(n) \leq C'' \cdot n^{\log_{\varphi}(2)}$, $C'' \in \mathbb{R}_{>0}$ for infinitely many $n \in \N$ and thus $\alpha \leq \log_{\varphi}(2)$.
\end{proposition}
\begin{proof}
 
 By Lemma \ref{lower}, we get
$n \leq \frac{\log_2(\gamma(a_n)) - \log_2(C)}{\log_2 (\varphi)}$ and hence
\begin{eqnarray*}\alpha(\gamma(a_n)) &\stackrel{\eqref{alpha}}{\leq}& \ell(a_n) \\ &\leq& C' \cdot 2^{n} \\
&\leq& C' \exp\left(\frac{\log(2) \cdot (\log_2(\gamma(a_n)) - \log_2(C))}{\log_2 (\varphi)} \right) \\
&=& C' \exp\left(\frac{-\log(2)\log_2(C)}{\log_2 (\varphi)} \right) \cdot (\gamma(a_n))^{\log_{\varphi}(2)}.
\end{eqnarray*}
This proves the claim.
\end{proof}
%\begin{remark}\label{rem4}
For any group $G$ and a word $w \in \mathbb{F}_2$, the word map $w:G \times G \rightarrow G$ is a natural map which is given by evaluation. We say that $w$ is a law for $G$ if the image of the corresponding word map is the identity element, i.e., $w(g,h)=1$ for all $g, h \in G$. By the method of Khalid Bou-Rabee \cite{MR2851069} and  \cite[Theorem 2.1]{Abdu} A Thom  \cite{Thom} proved that for $n \in \mathbb{N}$ there exists a word $w_n \in \mathbb{F}_2$ of length bounded by $O(\log(n)^{1.4411})$ which is a law for all nilpotent group of size at most $n$. Using the new upper bound of $\alpha$ in Theorem \ref{new}, we can improve this little bit:

\begin{corollary}\label{cor1}
Let $n \in \mathbb{N}$. There exists $w_n \in \mathbb{F}_2$ of length bound by $O(\log(n)^{\log_{\varphi}(2)})$ which is is a law for all nilpotent group of size at most $n$.
\end{corollary}
Moreover, following the method of A. Thom \cite{Thom}, we can improve his result \cite[Proposition 3.2]{Thom} on solvable groups little bit:

\begin{corollary}\label{cor2}
Let $n \in \mathbb{N}$. There exists $w_n \in \mathbb{F}_2$ of length bound by $O((\log(n))^{(\log_{\varphi}(2) +2.890457)})$ which is a law for all solvable groups of size at most $n$.
\end{corollary}
%\end{remark}

\begin{remark}\label{rem1}

We have $a_0=b^{-1} \in \gamma_1(\mathbb{F}_2)$, $a_1=[b^{-1},a] \in \gamma_2(\mathbb{F}_2)$, $a_2=[b^{-1},aba^{-1}]=[b^{-1},a][b,a]=[[b^{-1},a],b] \in \gamma_3(\mathbb{F}_2)$, and $a_3=[a_1,b_1]=[[b^{-1},a],[b,a]]\stackrel{\eqref{dep}}{=}[[b^{-1},a][b,a],[b,a]]=[[b^{-1},a],b],[b,a]]  \in \gamma_5(\mathbb{F}_2)$. So according to the construction $\gamma(a_n) \geq \gamma(a_{n-1})+\gamma(a_{n-2})$ you can easily see that $a_n \in \gamma_{f_{n+2}}(\mathbb{F}_2)$, where $f_m$ is the $m$-th Fibonacci number and is given by formula $f_m=\frac{\varphi^m-(-\varphi)^m}{\sqrt{5}}, m \in \mathbb{N}$.
\end{remark}
In the following two remarks we show the length of the shortest element in some central  series of this new construction which suggests the conjecture \ref{conj}. 
\begin{remark}\label{rem2}
We have:
\begin{enumerate}
\item $\min\{\ell (w): w\in \gamma_1(\mathbb{F}_2) \setminus \{1\} \}=1$.
\item  $\min\{\ell (w): w\in \gamma_2(\mathbb{F}_2) \setminus \{1\} \}=4$. Indeed $w=a^{n_1}b^{m_1}...a^{n_k}b^{m_k} \in \gamma_2(\mathbb{F}_2)$ if and only if $\sum_i{n_i} = \sum_i{m_i}=0$. So, the shortest non-trivial element that satisfies this condition is of length $4$.
\item\label{short} $\min\{\ell (w): w\in \gamma_3(\mathbb{F}_2 \setminus \{1\} \}=8$. Indeed if we consider any element of length $6$ in  $\gamma_2(\mathbb{F}_2)$ (which are finitely many elements) $[b^{-1},a][a,b]=[b^{-1},a][b,a][a,b]^2=[[b^{-1},a],b]][a,b]^2 \in \gamma_2(\mathbb{F}_2) \setminus \gamma_3(\mathbb{F}_2)$ since $[a,b]^2 \in \gamma_2(\mathbb{F}_2)\setminus \gamma_3(\mathbb{F}_2)$. Also if we consider the quotient $\gamma_2(\mathbb{F}_2)/ \gamma_3(\mathbb{F}_2)$ (which is free abelian group of finite rank, see from example \cite{quotient}, and \cite{lyndonschupp2}), then one can easily see that $[a,b]^2$ is the only factor of $[[b^{-1},a],b]][a,b]^2$ that survives in this quotient, hence  $[[b^{-1},a],b]][a,b]^2 \notin \gamma_3(\mathbb{F}_2)$.
 \end{enumerate}
\end{remark}
\begin{remark}\label{rem3}
It follows from Lemma \ref{depth} that the word $a_4=[[b^{-1},a][b,a],[a,b^{-1}][a,b]] \in \gamma_8(\mathbb{F}_2), \ell(a_4)=30$ and this suggests asymptotic behaviour of the form $O(n^{\nu})$, where $\nu=0.6113.. $. For the word $w=[[b^{-1},a][a,b],[a,b^{-1}][b,a]]$, we have that $\ell(w)=28 < \ell(a_4)=30$ but it follows from Lemma \ref{depth} and Remark \ref{rem2}, (\ref{short}) that $w \in \gamma_7(\mathbb{F}_2)$ since $\ell([b^{-1},a][a,b])=6$ (so $[b^{-1},a][a,b] \in \gamma_2(\mathbb{F}_2) \setminus \gamma_3(\mathbb{F}_2)$) and  $[a,b^{-1}][b,a]=[a,b]([b,a][a,b^{-1}])[b,a] \in \gamma_2(\mathbb{F}_2) \setminus \gamma_3(\mathbb{F}_2)$. Thus the word $w=[[b^{-1},a][a,b],[a,b^{-1}][b,a]]$ suggests the asymptotic behaviour of the form $O(n^{\mu})$, where $\mu=0.583..< \nu$.
\end{remark}

\begin{remark}
We can also consider this construction, we set $a'_0=a, \  b'_0=b,$ and define recursively:
$$a'_n=a'_{n-1}b'_{n-1}, \ \ b'_n=a'^{-1}_{n-1}b'^{-1}_{n-1}.$$ %\ \text{and} \ w'_n=[a'_n,b'_n].$$
You can easily see that the products here involve no cancellations and $\ell(a'_n)=2\ell(a'_{n-1})$ and also $\gamma(a'_n) \geq \gamma(a'_{n-1}) +\gamma(a'_{n-2})$ hence this construction suggests the asymptotic behaviour of the form $O(n^{\log_{\varphi}(2)})$.
\end{remark}

\begin{remark} For a subgroup $\Gamma \subset \mathbb{F}_2$, we define $\text{girth}(\Gamma):=\min\{\ell(w)|w\in \Gamma \setminus \{e\} \}$. The Author and A. Thom \cite{Abdu} proved that for $\Gamma \subset \mathbb{F}_2$ is a normal subgroup. Then the following holds: ${\rm girth}([\Gamma,\Gamma])\geq3.{ \rm girth}(\Gamma).$
We have $a_3=[[b^{-1},a],b],[a,b]], \ell(w_1)=14$, then ${\rm girth}([\gamma_3(\mathbb{F}_2),\gamma_2(\mathbb{F}_2)) \leq 14$. It follows from Remark \ref{rem2} that ${\rm girth}(\gamma_3(\mathbb{F}_2))=8, \  \text{and}  \ {\rm girth}(\gamma_2(\mathbb{F}_2))=4$, then we can easily see for two different normal subgroups that
$${\rm girth}([\gamma_3(\mathbb{F}_2),\gamma_2(\mathbb{F}_2)]) \leq 14 < 2.{\rm girth}(\gamma_3(\mathbb{F}_2))$$
and
$${\rm girth}([\gamma_3(\mathbb{F}_2),\gamma_2(\mathbb{F}_2)]) \leq 14 > 3.{\rm girth}(\gamma_2(\mathbb{F}_2)).$$
\end{remark}

%\section{Some applications}\label{app}

%\begin{theorem} We have
%$\exp(n^{\delta}) \preceq  F^{\rm nil}_{\mathbb F_2}(n)$ with
%with $$\delta =\log_2(\varphi)= 0,69491... .$$
%\end{theorem}
%\begin{proof}
%The proof is identical to the proof of Claim 1 on page 705 of \cite{MR2851069}. \end{proof}
\section{Almost laws for compact groups}\label{almost}

 Consider the word map on ${\rm SU}(k)$ for $w \in \mathbb{F}_2$ where ${\rm SU}(k)$ is the special unitary group. A. Thom \cite{MR3043070} proved that there exists a sequence of nontrivial  elements $(w_n)$ in $\mathbb{F}_2$, such that for every neighborhood $U \subset {\rm SU}(k)$ of identity, there exists $N \in \mathbb{N}$ such that $w_n({\rm SU}(k) \times {\rm SU}(k)) \subset U$ for all $n \geq N$. This sequence of words is called almost law for compact groups see, \cite{breu} for more details. For $u,v \in {\rm SU}(k)$, there is a metric $d(u,v)  := \|u-v\|$ where $\|.\|$ is the operator norm. We set  $L_k(w) := \max \{ d(1_k,w(u,v)) \mid u,v \in {\rm SU}(k) \}.$ The author and A. Thom \cite{Abdu} proved that there exists an almost law $(w_n)$ for ${\rm SU}(k)$ such that there exists a constant $C>0$ depending on $k$ such that
$L_k(w_n)  \leq \exp\left( -C \cdot \ell(w_n)^{\delta} \right)$
with $\delta = \frac{\log_2(1 + \sqrt{2})}{\log_2(3+\sqrt{17}) - 1}= 0.69391... .$
Using this new construction we can improve this little bit.

\begin{theorem} \label{almost} Let $k \in \N$. There exists an almost law $(w_n)_n$ for ${\rm SU}(k)$ such that the following holds
there exists a constant $C>0$ such that
$$L_k(w_n)  \leq \exp\left( -C \cdot \ell(w_n)^{\log_2 (\varphi)}\right)$$
where $\varphi$ is the golden ratio.
\end{theorem}
\begin{proof}
%Our basic method is a well-known contraction property of the commutator map in a Banach algebra. Let $k$ be fixed. In terms of the function $L_k$,
It follows from \cite[Lemma 2.1]{MR3043070}, that 
\begin{equation}\label{contact}
\|1-u_1u_2\| \leq 2\|1-u_1\| \|1-u_2\|, \ \text{for } u_1, u_2  \in {\rm SU}(k).
\end{equation}
%\begin{equation} \label{zassenhaus}
%L_k([w,v]) \leq 2 \cdot L_k(w)L_k(v).
%\end{equation} 
By \cite[Corollary 3.3.]{MR3043070}  there exist words $w,v \in \mathbb F_2$ which generate a free subgroup and satisfy $L_k(w),L_k(v) \leq \frac13$. Let us set $w_n:=a_n(w,v)$. It follows from Lemma \ref{depth} that 
\begin{equation}  \label{length}
\ell(w_n) \leq C'' \cdot 2^n
\end{equation} 

for some constant $C''>0$. On the other side, Equation \eqref{contact} and the equation $a_{n+2} = [a_{n+1},a_{n-1}]$ show that $L_k(w_n) \leq 2 \cdot L_k(w_{n-1}) L_k(w_{n-2})$ or equivalently
$$- \log(2L_k(w_n)) \geq - \log(2L_k(w_{n-1})) - \log(2L_k(w_{n-2})).$$
Thus there exists a constant $D>0$ (as in the proof of Lemma \ref{lower}) such that
\begin{equation} \label{small}
- \log (2L_k(w_n)) \geq D \cdot \varphi^n,
\end{equation}
for some constant $D>0$. Hence,
$$L_k(w_n) \stackrel{\eqref{small}}{\leq} \frac12 \exp\left(- D \cdot \varphi^n \right) \stackrel{\eqref{length}}{\leq} \exp\left(- C  \cdot \ell(w_n)^{\log_2 (\varphi)} \right)$$
for some constant $C$. This proves the claim.\end{proof}

%It would be interesting to find a more direct relationship between the growth of the girth of the lower central series and the asymptotics encountered in Theorem \ref{almost}. It is presently unclear if $1 + \varepsilon$ for any $\varepsilon>0$ (or even for $\varepsilon=0$) is enough in Theorem \ref{almost}, see also Section 5.4 in \cite{breu} for a discussion of this question.

\section*{Acknowledgments}
We want to thank Andreas Thom and Jakob Schneider for interesting comments and the MPI-MIS Leipzig for support and an excellent research environment.


\begin{bibdiv}
\begin{biblist}

\bib{breu}{article}{
   author={Menny Aka},
   author={Emmanuel Breuillard},
   author={Lior Rosenzweig},
   author={Nicolas de Saxc\'e},
   title={Diophantine properties of nilpotent Lie groups},
   journal={Compositio Mathematica},
   volume={151},
      date={2015},
      number={6},
      pages={1157--118},
}

\bib{MR2851069}{article}{
   author={Bou-Rabee, Khalid},
   title={Approximating a group by its solvable quotients},
   journal={New York J. Math.},
   volume={17},
   date={2011},
   pages={699--712},
   
}



%\bib{MR2583614}{article}{
   %author={Buskin, Nikolai Vladislavovich},
   %title={Efficient separability in free groups},
  % language={Russian, with Russian summary},
  % journal={Sibirsk. Mat. Zh.},
%   volume={50},
%   date={2009},
 %  number={4},
 %  pages={765--771},
 %  translation={
 %     journal={Sib. Math. J.},
%      volume={50},
     % date={2009},
   %   number={4},
    %  pages={603--608},
  % },
%}


%\bib{fox}{article}{
 %  author={Fox, Ralph H.},
%   title={Free differential calculus. I. Derivation in the free group ring},
 %  journal={Ann. of Math. (2)},
%   volume={57},
%   date={1953},
%   pages={547--560},
%}

%\bib{MR2784792}{article}{
  % author={Kassabov, Martin},
   %author={Matucci, Francesco},
  % title={Bounding the residual finiteness of free groups},
  % journal={Proc. Amer. Math. Soc.},
 %  volume={139},
 %  date={2011},
 %  number={7},
  % pages={2281--2286},
%}
%\bib{levi1}{article}{
%   author={Levi, Friedrich},
 %  title={\"Uber die Untergruppen freier Gruppen I},
 %  journal={Math. Z.},
 %  volume={32},
%   date={1930},
  % number={1},
  % pages={315--318},
%}

%\bib{levi2}{article}{
%   author={Levi, Friedrich},
  % title={\"Uber die Untergruppen der freien Gruppen II},
  % journal={Math. Z.},
  % volume={37},
  % date={1933},
 %  number={1},
   %pages={90--97},
%}

\bib{lyndonschupp}{book}{
   author={Lyndon, Roger C.},
   author={Schupp, Paul E.},
   title={Combinatorial group theory},
   note={Ergebnisse der Mathematik und ihrer Grenzgebiete, Band 89},
   publisher={Springer-Verlag},
   place={Berlin},
   date={1977},
   pages={xiv+339},
 }
\bib{lyndonschupp2}{book}{
   author={Magnus},
   author={Karrass},
   author={Soltar},
   title={Combinatorial group theory},
   %note={Ergebnisse der Mathematik und ihrer Grenzgebiete, Band 89},
   publisher={John Wiley and Sons, Inc},
   place={New York},
   date={1966},
   pages={xii+444},
 }
\bib{MR2737679}{article}{
   author={Malestein, Justin},
   author={Putman, Andrew},
   title={On the self-intersections of curves deep in the lower central
   series of a surface group},
   journal={Geom. Dedicata},
   volume={149},
   date={2010},
   pages={73--84},
}

%\bib{MR2970452}{article}{
%   author={Rivin, Igor},
%   title={Geodesics with one self-intersection, and other stories},
 %  journal={Adv. Math.},
 %  volume={231},
 %  date={2012},
 %  number={5},
  % pages={2391--2412},
%}

\bib{MR3043070}{article}{
   author={Thom, Andreas},
   title={Convergent sequences in discrete groups},
   journal={Canad. Math. Bull.},
   volume={56},
   date={2013},
   number={2},
   pages={424--433},
}

\bib{Abdu}{article}{
   author={Elkasapy, Abdelrhman}
   author={Thom, Andreas},
   title={On the length of the shortest non-trivial element in
the derived and the lower central series},
   journal={J. Group theory},
   volume={18},
   date={2015},
   number={7},
   pages={793--804},
}

\bib{Thom}{article}{
   author={Thom, Andreas},
   title={About the length of laws for finite groups},
   journal={arXiv preprint
arXiv:1508.07730; to appear in Israel Journal of Mathematics (2015)},
 }

\bib{quotient}{article}{
  author={Anthony M. Gagalione},
  title={Factor groups of the lower central series for special products},
  journal={Journal of Algebra},
  volume={37},
  date={1975},
  pages={172--185},
}

%\bib{Dan}{book}{
  % author={Dan Segal},
   %author={Schupp, Paul E.},
 %  title={Words},
   %note={Ergebnisse der Mathematik und ihrer Grenzgebiete, Band 89},
  % publisher={Cambridge University Press},
 %  place={United Kingdom},
%   date={2009},
 %  pages={121},
% }



\end{biblist}
\end{bibdiv}
\end{document}